\documentclass[11pt,reqno]{amsart}
\usepackage[all]{xy}
\usepackage[latin1]{inputenc} 
\usepackage[T1]{fontenc}
\usepackage[dvips]{graphics,graphicx}
\usepackage{amsfonts,mathabx}
\usepackage{amssymb}
\usepackage{amsmath}
\usepackage{mathrsfs, mathtools}
\usepackage{esint}
\usepackage{array, makecell} 
\allowdisplaybreaks

\usepackage{amsbsy,
	amsopn,
	amscd, 
	amsxtra, 
	amsthm,
	verbatim}
\usepackage{upref}
\usepackage[dvipsnames]{xcolor}
\usepackage{enumerate}
\usepackage[colorlinks,
linkcolor=red,
anchorcolor=red,
citecolor=red
]{hyperref}
\usepackage{cleveref}

\usepackage{thmtools}
\usepackage{caption}
\usepackage{subcaption}
\usepackage{verbatim}
\usepackage{booktabs}
\usepackage{longtable}

\newtheorem{theorem}{Theorem}
\newtheorem{definition}{Definition}
\newtheorem{lemma}{Lemma}[section]
\newtheorem{proposition}[lemma]{Proposition}
\newtheorem*{proposition*}{Proposition}
\newtheorem{corollary}[theorem]{Corollary}
\newtheorem*{corollary*}{Corollary}
\newtheorem{assumption}{Assumption}

\theoremstyle{definition}
\newtheorem{remark}[lemma]{Remark}
\newtheorem{example}[lemma]{Example}

\usepackage{geometry}
\usepackage{float}
\newcommand{\vertiii}[1]{{\left\vert\kern-0.25ex\left\vert\kern-0.25ex\left\vert #1 \right\vert\kern-0.25ex\right\vert\kern-0.25ex\right\vert}}

\makeatletter
\newcommand*{\rom}[1]{\expandafter\@slowromancap\romannumeral #1@}
\makeatother

\newcommand{\opT}{\mathcal{T}}

\newcommand{\tdz}{\tilde{z}}
\DeclareMathOperator*{\argmin}{arg\,min}

\newcommand{\dif}{\mathop{}\!\mathrm{d}}

\newcommand{\R}{\mathbb{R}}

\newcommand{\Lap}{\mathrm{L}_x}

\newcommand{\eps}{\epsilon}
\newcommand{\abs}[1]{\lvert#1\rvert}

\newcommand{\norm}[1]{\lVert#1\rVert}

\renewcommand{\Re}{\mathfrak{Re}}

\newlength{\leftstackrelawd}
\newlength{\leftstackrelbwd}
\def\leftstackrel#1#2{\settowidth{\leftstackrelawd}%
{${{}^{#1}}$}\settowidth{\leftstackrelbwd}{$#2$}%
\addtolength{\leftstackrelawd}{-\leftstackrelbwd}%
\leavevmode\ifthenelse{\lengthtest{\leftstackrelawd>0pt}}%
{\kern-.5\leftstackrelawd}{}\mathrel{\mathop{#2}\limits^{#1}}}

\def\bigl{\mathopen\big}

\def\bigr{\mathclose\big}

\newcommand{\opL}{\mathcal{L}}
\newcommand{\oplnh}{\mathcal{L}_{\mathrm{NH}}}

\newcommand{\Lv}{\mathrm{L}_v}
\newcommand{\calX}{\mathcal{X}}
\newcommand{\calV}{\mathcal{V}}

\newcommand{\rel}   {\mathrm{rel}}
\DeclareMathOperator{\gap}  {gap}
\DeclareMathOperator{\spec} {spec}

\title{Hypocoercivity meets lifts}
\author{Giovanni Brigati}
\address{Institute of Science and Technology Austria, Am Campus 1, Klosterneuburg, 3400 Austria}
\email{giovanni.brigati@ist.ac.at}
\author{Francis L\"orler}
\address{Institute for Applied Mathematics, University of Bonn. Endenicher Allee 60, 53115 Bonn, Germany}
\email{loerler@uni-bonn.de}
\author{Lihan Wang} \address{Department of Mathematical Sciences, Carnegie Mellon University, Pittsburgh, PA, 15213 USA}
\email{lihanw@andrew.cmu.edu}
\date{\today}

\begin{document}

\begin{abstract}
  We unify the variational hypocoercivity framework established by D.~Albritton, S.~Armstrong, J.-C.~Mourrat, and M.~Novack \cite{albritton2024variational}, with the notion of second-order lifts of reversible diffusion processes, recently introduced by A.~Eberle and the second author \cite{eberle2024non}. We give an abstract, yet fully constructive, presentation of the theory, 
  so that it can be applied to a large class of linear kinetic equations. As this hypocoercivity technique does not twist the reference norm, we can recover accurate and sharp convergence rates in various models. Among those, adaptive Langevin dynamics (ALD) is discussed in full detail and we show that for near-quadratic potentials, with suitable choices of parameters, it is a near-optimal second-order lift of the overdamped Langevin dynamics. As a further consequence, we observe that the Generalised Langevin Equation (GLE) is a also a second-order lift, as the standard (kinetic) Langevin dynamics are, of the overdamped Langevin dynamics. Then, convergence of (GLE) cannot exceed ballistic speed, i.e.~the square root of the rate of the overdamped regime. We illustrate this phenomenon with explicit computations in a benchmark Gaussian case.
\end{abstract}
\maketitle

\section{Introduction}\label{sec:intro}
\subsection{Setting and main goal}
In this note, we obtain long-time convergence results to a large class of linear kinetic equations that can be written as
\begin{equation}\label{eq:genkineq}
    \partial_t f + \opT f = \Lv f, \quad f =f(t,x,v), \quad (t,x,v) \in \R_+\times\calX \times \calV.
\end{equation}
Here $\opT$ is the transport operator, and the operator $\Lv$ describes the energy dissipation in the $v$-variable, with precise assumptions and examples stated below. Usually $x$ and $v$ are interpreted as position and velocity variables, respectively, although the situations we consider are general, and $\calX$ and $\calV$ denote general measurable state spaces that may have different dimensions. 

\smallskip

We focus in particular on scenarios where \eqref{eq:genkineq} has a unique normalised steady state of the form
\begin{equation}\label{eq:geniv}
     \hat{\mu}(\dif x \, \dif v) =  \mu(\dif x) \kappa(x,\dif v),
\end{equation}
where $\hat{\mu}$ is a probability distribution, $\mu$ is the marginal distribution of $\hat{\mu}$ on $\calX$, and $\kappa$ denotes a Markov kernel. In many examples, $\hat\mu$ is a product measure $\hat{\mu} = \mu \otimes \kappa$ where $\kappa$ is independent of $x$, but this framework does not require $\hat{\mu}$ to be a product measure; in fact it can also be applied to nonequilibrium dynamics, where only the existence but not the expression of $\hat{\mu}$ is known \cite{dietert20232}.
The primary goal of this note is to highlight the link between variational hypocoercivity and non-reversible lifts, in a unified language, which is extensively explained in $\S$\ref{ss:fhtl}-\ref{ss:ass}.
Then, we state our \textbf{main result} Theorem \ref{thm:expconv} in $\S\ref{ss:mr}$, which yields a hypocoercive estimate in the form   
\begin{equation}
    \label{informal}
\|f(t,\cdot,\cdot)\|_{L^2(\hat{\mu})} \le C\mathrm{e}^{-\lambda t} \|f(t=0,\cdot,\cdot) \|_{L^2(\hat{\mu})}, 
\end{equation}
for all solutions $f$ to \eqref{eq:genkineq} with zero average initial data. We construct a fully explicit, accurate decay rate $\lambda>0$, and an explicit constant $C>1$.
In \S\ref{ss:ex}, we list a few important examples which are covered by our analysis.

\subsection{From hypocoercivity to non-reversible lifts}\label{ss:fhtl} In many cases of interest, the operator $\Lv$ is degenerate dissipative, i.e.~$\Lv \leq 0$, but $\ker(\Lv)$ is larger than $\ker(\Lv - \opT)$. In other words, equation \eqref{eq:genkineq} admits local equilibria which are not global ones. 
Convergence to global equilibrium still holds, as long as the conservative operator $\opT$ is able to force the solution to \eqref{eq:genkineq} outside $\mathrm{Ker} \, \Lv$. This is the heart of hypocoercivity, stemming from \cite{kolmogoroff1934zufallige,hormander1967hypoelliptic}, and C.~Villani's monograph \cite{villani2009hypocoercivity}. To counter the degeneracy of $\Lv$, a modified $H^1$-norm $\tilde{\mathcal E}$ is built in \cite{villani2009hypocoercivity}, such that $\tilde{\mathcal E} \approx \|\cdot\|_{H^1(\dif x\,\dif v)}$, and $\tilde{\mathcal E}$ acts as a strict entropy functional for \eqref{eq:genkineq}, i.e.\ for all solutions $f$, and a suitable global equilibrium $f_\infty$, we have $\frac{\dif}{\dif t} \tilde{\mathcal E} (f(t,\cdot,\cdot)-f_\infty) \leq - \lambda \, \tilde{\mathcal E}(f(t,\cdot,\cdot)-f_\infty)$, for some $\lambda>0$. Thus, for a constant $C>1$, 
\begin{equation}
    \label{hypo}
    \|f(t,\cdot,\cdot)-f_\infty\|_{H^1(\dif x\,\dif v)} \leq C \, \mathrm{e}^{-\lambda \, t} \|f(0,\cdot,\cdot)-f_\infty\|_{H^1(\dif x\,\dif v)}.
\end{equation}
This scheme has been fruitfully exploited in many works, including \cite{baudoin_bakry-emery_2013}, where a twisted $\Gamma_2$ calculus, a combination of \cite{bakry2006diffusions} and \cite{villani2009hypocoercivity}, is established. We refer the reader to \cite{bernard2022hypocoercivity} for a review of the literature. Later on, an $L^2$-hypocoercivty approach, inspired by diffusion limits, was outlined by Dolbeault, Mouhot, and Schmeiser in \cite{dolbeault2015hypocoercivity}, with ideas from earlier works \cite{herau_hypocoercivity_2006, dolbeault_hypocoercivity_2009}. In this influential work, an abstract framework is laid out with several assumptions that can be directly verified for a large class of kinetic equations (even in presence of non-regularising effects, such as scattering), with a wide range of applications. 
Inequalities in the form \eqref{hypo}, and extension of those in case of weak confinement \cite{bouin2021hypocoercivity,bouin2022fractional,bouin20222}, are obtained in the $L^2$-norm. 

\smallskip
As the $L^2$-framework of hypocoercivity provides explicit estimates for the constants $C>1, \lambda>0$, the problem of finding optimal values is posed. This can be done via mode-by-mode Fourier analysis in special situations \cite{arnold2021sharpening}. In general, if providing good estimates for the rate $\lambda$ is delicate on its own, optimising $\lambda$ and $C$ at the same time is problematic, and the two constants can even compete \cite{achleitner2017optimal}. 
\smallskip

This motivates a new \emph{variational} approach to $L^2$-hypocoercivity, where the $L^2$-norm is not twisted, first established by Albritton, Armstrong, Mourrat and Novack in \cite{albritton2024variational} and later developed in \cite{cao2023explicit, lu2022explicit, dietert2022quantitative,   brigati2023time, brigati2023construct,  dietert20232,  brigati2024explicit}. The method relies on the \emph{time-averaged} $L^2$-energy,  
\begin{equation*}
       \mathcal{H}(t) = \fint_t^{t+T}\|f(s,\cdot,\cdot)\|^2_{L^2(\hat{\mu})} \dif s,
\end{equation*}
which can be related to its dissipation along \eqref{eq:genkineq} 
\begin{equation*}
    \mathcal{I}(t) := -\frac{\dif }{\dif t} \mathcal{H}(t) = \fint_t^{t+T} \frac{\dif }{\dif s} \|f(s,\cdot,\cdot)\|^2_{L^2(\hat{\mu})} \dif s,
\end{equation*}
via an adapted Poincar\'e inequality
\begin{equation}\label{abstractpoi}
    \mathcal{I}(t) \geq \lambda_T \, \mathcal{H}(t).
\end{equation}
As in the classical parabolic case,  exponential decay for $\mathcal{H}(t)$ is recovered. This directly implies a hypocoercivity estimate \eqref{hypo}, see Theorem \ref{thm:expconv}, with the same rate $\lambda_T$.
The leading idea is that the time average allows the transport operator to carry relaxation from the $v$ to the $x$ variables, which is a non-instantaneous effect, as witnessed also in the kinetic De Giorgi theory \cite{anceschi2024poincar}. The parameter $T$ can be tuned to optimise the rate $\lambda_T$, without enlarging any other constant. 
\smallskip

The adapted Poincar\'e inequality \eqref{abstractpoi} is usually proved in two steps, namely a $(t,x)$-Poincar\'e--Lions inequality for the velocity average $\Pi f$ defined in \eqref{eq:defpiv}, and an averaging lemma (variants of which can be traced back to at least \cite{golse1988regularity}), relating weak norms of $\Pi f$ to the dissipation term $\mathcal{I}(t)$.
It was unclear whether this was the only reasonable path leading to \eqref{abstractpoi}, until the results of \cite{eberle2024non}. In that contribution, a new structure of \emph{non-reversible lifts of reversible diffusion processes} is formulated, of which \eqref{eq:genkineq} is a particular case under our assumptions below. The abstract theory of \cite{eberle2024non}, once rewritten for \eqref{eq:genkineq}, is reminiscent of \cite{albritton2024variational}. A novelty is that the intermediate steps for \eqref{abstractpoi} (i.e.~the Poincar\'e--Lions inequality and the averaging lemma) are endowed with a clear interpretation and identified by the lift structure, as highlighted by the proof of Theorem \ref{thm:expconv}. As an additional result, the work \cite{eberle2024non} allows to derive both upper and lower abstract bounds for the rate $\lambda_T$, which match well with the concrete computations of \cite{brigati2023construct,cao2023explicit}.

\subsection{Assumptions and discussion}\label{ss:ass} Following the lift structure, we provide a general framework for the variational approach \cite{albritton2024variational}, adapting the notation and the style of presentation of the DMS (Dolbeault--Mouhot--Schmeiser) framework  \cite{dolbeault2015hypocoercivity}. 

For any function $f(t, x,v)$, we define its projection
\begin{equation}\label{eq:defpiv}
    \Pi f(t, x) := \int_{\calV} f(t, x,v)  \kappa(x, \dif v),
\end{equation}
which yields a self-adjoint operator on $L^2(\hat{\mu})$.
We assume that the operator $\opL = \Lv - \opT$ is the generator of some Markov process with transition semigroup $(\hat{P}_t)_{t\ge 0}$, so that the dynamics \eqref{eq:genkineq} is mass-preserving. Moreover, we impose the following assumptions on $\opT$.

\begin{assumption}\label{ass:dmsh3}
$\Pi \opT \Pi =0$.
\end{assumption}
\begin{assumption}\label{ass:dmsh2}
    The operator $\Lap := -(\opT \Pi)^*(\opT \Pi)$ is self-adjoint with respect to $L^2(\mu)$ and corresponds to the generator of some reversible diffusion process with transition semigroup $(P_t)_{t\ge 0}$. Moreover, we have the Poincar\'e inequality (macroscopic coercivity)
    \begin{equation}\label{eq:pix}
        \mathrm{P}_x := \inf_{f\in L_0^2(\hat\mu)\setminus\{0\}}\frac{\norm{\opT\Pi f}_{L^2(\hat\mu)}^2}{\norm{\Pi f}_{L^2(\mu)}^2}>0,
    \end{equation}
    or, equivalently, $\Lap$ has spectral gap $\mathrm{P}_x$.
\end{assumption}

\begin{definition}[{\cite[Definition 1]{eberle2024non}}]\label{def:lift}
The semigroup $(\hat{P}_t)_{t\geq 0}$ is a second-order lift of $(P_{t/2})_{t\geq0}$ if Assumptions \ref{ass:dmsh3}-\ref{ass:dmsh2} are satisfied.
\end{definition}

Let us comment on this definition. Motivated by an acceleration in convergence to equilibrium through non-reversibility, in discrete time, a concept of lifts of Markov chains was introduced by \cite{Diaconis2000Lift} and formalised by \cite{chen1999Lift}. The concept of second-order lifts introduced in \cite{eberle2024non} is related yet distinct, as it covers continuous-time dynamics and incorporates a linear-to-quadratic time rescaling which is crucial in the kinetic continuous-time case. Intuitively, under minor additional assumptions, see \cite[Remark 2 (v)]{eberle2024non}, Definition~\ref{def:lift} corresponds to
\begin{equation*}
    \int_\calV \hat P_t(\Pi f)(x,v)\kappa(x,\dif v) = P_{(t/2)^2}\Pi f(x) + o(t^2)\quad\text{for all }f\in L^2(\hat\mu)
\end{equation*}
as $t\to 0$, i.e.\ the semigroup $P_{(t/2)^2}$ behaves approximately as an averaged version of the lifted transition semigroup $\hat P_t$ for small values of $t$.

\begin{remark}[Lower bounds for lifts]\label{rem:lowerbound}
    Using the framework of lifts, under Assumptions \ref{ass:dmsh3}-\ref{ass:dmsh2} only, we can recover a general lower bound for the convergence rate of \eqref{eq:genkineq}.
    Denoting the relaxation time of $(\hat P_t)_{t\geq 0}$ as 
    \begin{equation*}
        t_\rel(\hat P) = \inf\{t\geq0\colon\norm{\hat P_tf}_{L^2(\hat\mu)}\leq \mathrm{e}^{-1}\norm{f}_{L^2(\hat\mu)}\textup{ for all }f\in L_0^2(\hat\mu)\},
    \end{equation*}
    the fact that $(\hat P_t)_{t\geq 0}$ is a second-order lift of $(P_{t/2})_{t\geq 0}$ immediately yields the lower bound
    \begin{equation*}
        t_\rel(\hat P)\geq\frac{1}{\sqrt{2}}\mathrm{P}_x^{-1/2}
    \end{equation*}
    by \cite[Theorem 11]{eberle2024non}. In particular, if \eqref{hypo} is satisfied for some $C>1$ and $\lambda>0$, then the asymptotic decay rate $\lambda$ can be bounded as
    \begin{equation*}
        \lambda\leq (1+\log C)\sqrt{\mathrm{P}_x}.
    \end{equation*}
    Since $\mathrm{P}_x^{-1}$ is the relaxation time of the reversible transition semigroup $(P_t)_{t\geq 0}$, this shows that convergence to equilibrium as measured by the relaxation time cannot be accelerated by more than a square root through lifting.
\end{remark}
The following Assumptions \ref{ass:opT}-\ref{ass:lions} are not needed for $(\hat{P}_t)$ to be a second-order lift of $(P_t)$, and therefore only affect the upper bound of the convergence analysis.

\begin{assumption}[{\cite[Assumption 1 (ii)]{eberle2024non}}]\label{ass:opT}
     The operator $\opT$ is non-negative in the sense that for any $f\in \mathrm{Dom}(\opT)$, 
\begin{equation}\label{eq:optpos}
    \iint_{\calX \times \calV} f(x,v) \opT f(x,v) \dif \hat{\mu}(x,v) \ge 0.
\end{equation}
\end{assumption}
Assumption \ref{ass:opT} holds automatically with equality when $\opT$ is antisymmetric, but follows from a direct calculation in some other examples. 
We define \[\|f\|_{L^2(\mu;H^1(\kappa))}^2 := \iint_{\calX\times \calV} (|f|^2 + |\nabla_v f|^2) \dif \hat{\mu}(x,v),\]
and $\|f\|_{L^2(\mu;H^{-1}(\kappa))}$ by $L^2$-duality
\[ \|f\|_{L^2(\mu;H^{-1}(\kappa))} := \sup_{\|g\|_{L^2(\mu;H^1(\kappa))}=1} \iint_{\calX \times \calV} fg\dif \hat{\mu}(x,v).\]
\begin{assumption}\label{ass:dmsh1}
The operator $\Lv$ is self-adjoint in $L^2(\hat{\mu})$. In addition, its nullspace $\ker(\Lv)$ is spanned by functions that only depend on the $x$-variable, i.e.\ \[ \ker(\Lv) = \mathrm{Im}(\Pi).\]
Moreover, we have microscopic coercivity \cite[Assumption (H1)]{dolbeault2015hypocoercivity}:
     \begin{equation}\label{eq:piv}
       -\langle f, \Lv f\rangle_{L^2(\hat{\mu})} \ge \mathrm{P}_v \|(\mathrm{Id}-\Pi) f\|^2_{L^2(\hat{\mu})}.
    \end{equation}
    Finally, 
    \begin{equation}\label{eq:H-1kappa}
        \|\Lv f\|_{L^2(\mu;H^{-1}(\kappa))}^2 \le -R\langle f, \Lv f\rangle_{L^2(\hat{\mu})}. 
    \end{equation}
\end{assumption} 

\begin{remark}[DMS Hypocoercivity and lifts] \label{rmk:dmslift}
    Assumptions \ref{ass:dmsh3}-\ref{ass:dmsh2} correspond to assumptions (H3) and (H2) of the DMS method \cite{dolbeault2015hypocoercivity} respectively. At the same time, Assumptions \ref{ass:dmsh3} and \ref{ass:dmsh2} are precisely the two assumptions in \cite{eberle2024non} which indicate that $(\hat{P}_t)_{t\geq0}$ is a second-order lift of $(P_{t/2})_{t\geq0}$ as defined in \cite[Definiton 1]{eberle2024non}, aside a from a domain condition. If we disregard \eqref{eq:H-1kappa}, Assumption \ref{ass:dmsh1} is precisely \cite[Assumption (H1)]{dolbeault2015hypocoercivity}. The inequality \eqref{eq:H-1kappa} holds true, with $R=\gamma$, under the most common choices $\Lv = \gamma (\Pi-\mathrm{Id})$ or $\Lv = -\gamma \nabla_v^*\nabla_v$. The coercivity assumptions \eqref{eq:pix} and \eqref{eq:piv} can be weakened into some weak or weighted version of the Poincar\'e inequality, see \cite{brigati2023time, brigati2023construct, brigati2024explicit, dietert20232}. 
\end{remark}

As for our final assumption, readers familiar with the works \cite{cao2023explicit, brigati2023construct, eberle2024non} will recognise that it resembles the solution of the divergence equation, which is in duality with the aforementioned Poincar\'e--Lions inequality. There are various ways in the literature to construct such solutions with required regularity estimates, see also \cite{eberle2024divergence,eberle2024rtp} and \cite{lehec2024convergence}.
\begin{assumption}\label{ass:lions}
    Let $T>0$, $\lambda (\dif t) := \frac{1}{T}\mathbf{1}_{[0,T]} \dif t$ denote the normalised Lebesgue measure on $[0,T]$, and denote $\bar{\mu} := \lambda \otimes \mu$. Then for any $g(t,x) \in L^2(\bar{\mu})$ with $\iint g(t,x) \dif \bar{\mu}(t,x) =0$, there exist functions $\phi_0,\phi_1$ with boundary conditions $\phi_0(0,\cdot) = \phi_0(T,\cdot)=0, \ \opT\phi_1(0,\cdot) = \opT\phi_1(T,\cdot) =0$ such that
    \begin{equation}\label{eq:generaldiveq}
    -\partial_t \phi_0 - \Lap\phi_1 = g.
\end{equation}
Moreover, there exist constants $C_{0,T}, C_{1,T}$ depending on $T$ such that $\phi_0,\phi_1$ satisfy the regularity estimates 
\begin{equation}\label{eq:divbdL2}
    \|\phi_0 + \opT \phi_1\|_{L^2({\bar{\mu}};H^{1}(\kappa))} \le C_{0,T}\|g\|_{L^2(\bar{\mu})},
\end{equation} and
\begin{equation}\label{eq:divbdH1}
     \|(\partial_t - \opT^*)(\phi_0 + \opT  \phi_1) \|_{L^2(\lambda\otimes\hat\mu)} \le C_{1,T}\|g\|_{L^2(\bar{\mu})}.
\end{equation}
\end{assumption}
In practice, for most of the examples we consider, $\kappa$ only depends on $v$ and is standard Gaussian, thus Assumption \ref{ass:lions} can be verified since any finite moments of $v$ are bounded, and the estimates \eqref{eq:divbdL2}, \eqref{eq:divbdH1} boil down to \[\|\phi_0\|_{H^1(\bar{\mu})} + \|\partial_t\phi_1\|_{H^1(\bar{\mu})} + \|\nabla_x\phi_1\|_{H^1(\bar{\mu})} \lesssim \|g\|_{L^2(\bar{\mu})}. \]
In previous works \cite{cao2023explicit, brigati2023construct, eberle2024non, lehec2024convergence}, this can be achieved if $\mu$ satisfies a Poincar\'e inequality \eqref{eq:pix}, and for $\dif \mu(x) \mathrel{\propto} \exp(-U(x)) \dif x$, the Hessian of $U$ satisfies a (not necessarily positive) lower bound. In fact, one may notice that in the previous works, the verification of Assumption \ref{ass:lions} is the only place where macroscopic coercivity \eqref{eq:pix} is used. We nevertheless write the two assumptions separately as \eqref{eq:pix} is crucial for the definition of lifts and obtaining convergence lower bounds.

\smallskip
\begin{remark}[Variational hypocoercivity and DMS]
In light of Remark \ref{rmk:dmslift}, it is natural to compare Assumption \ref{ass:lions} with \cite[Assumption (H4)]{dolbeault2015hypocoercivity}. In \cite[Section 2]{dolbeault2015hypocoercivity}, the authors show that \cite[Assumption (H4)]{dolbeault2015hypocoercivity} is related to elliptic regularity theory. On the other side, Assumption \ref{ass:lions} requires regularity for a divergence equation, which can be completed to an elliptic one \cite{cao2023explicit,brigati2023construct}. The two assumptions differ in their spirit, though, since  \cite[Assumption (H4)]{dolbeault2015hypocoercivity} is used to control error terms arising from a modified $L^2$-norm, whereas the bounds \eqref{eq:divbdL2}-\eqref{eq:divbdH1} arise from the lift structure via an integration by parts. 
\end{remark}
\smallskip

\subsection{Main result}\label{ss:mr}
Under the Assumptions \ref{ass:dmsh3}--\ref{ass:lions}, we can prove exponential convergence of the solutions of \eqref{eq:genkineq} to equilibrium.
\begin{theorem}\label{thm:expconv}
    Let $f(t,x,v)$ be the solution of \eqref{eq:genkineq} with initial condition $f(t=0,x,v) = f_0(x,v) \in L^2(\hat{\mu})$. Then as $t\to \infty$, $f_t$ converges exponentially fast to $(f_0):= \int f_0(x,v) \dif \hat{\mu}(x,v)$. More precisely, there exist constants $C>1,\lambda>0$ such that
    \[
        \|f(t,\cdot,\cdot) - (f_0) \|_{L^2(\hat{\mu})} \le Ce^{-\lambda t} \|f_0 - (f_0)\|_{L^2(\hat{\mu})} .
\]
In particular, suppose there exists some $T_*>0$ such that
\[ T_* = \argmin_{T>0}\{C_{1,T} + C_{0,T}\sqrt{R\mathrm{P}_v}\}, \]
then we have $C=\exp(T_*\lambda_{T_*})$ with the optimised choice of convergence rate 
\begin{equation}\label{eq:abstractrate}
    \lambda_{T_*} = \frac{2\mathrm{P}_v}{1 + (C_{1,T_*} + C_{0,T_*}\sqrt{R\mathrm{P}_v})^2},
\end{equation}
\end{theorem}
\begin{remark}
    To verify the optimality of the convergence rate in Theorem \ref{thm:expconv}, let us consider the Langevin dynamics \eqref{eq:ld}, for which  $R = \mathrm{P}_v = \gamma$. If $U$ is convex and has Poincar\'e constant $\mathrm{P}_x$, then the proof in \cite{eberle2024divergence} gives $C_{0,T} = C_0(T+\frac{1}{\sqrt{\mathrm{P}_x}})$ and $C_{1,T} = C_1(1+ \frac{1}{T\sqrt{\mathrm{P}_x}})$ where $C_0,C_1$ denote absolute constants which can be made explicit. Substituting these into \eqref{eq:abstractrate}, we get \[\lambda \sim \frac{2\gamma}{\inf_{T>0}\Big(1+ \frac{1}{T\sqrt{\mathrm{P}_x}}+T\gamma + \frac{\gamma}{\sqrt{\mathrm{P}_x}}\Big)^2}.\] Optimising first in $T_*\sim \mathrm{P}_x^{-1/4}\gamma^{-1/2}$ and then in $\gamma\sim \sqrt{\mathrm{P}_x}$, we have $\lambda  \sim \sqrt{\mathrm{P}_x}$ which recovers the result in \cite{cao2023explicit, eberle2024non}, and the constant $C=e^{\lambda T}$ is then bounded above by a universal constant. 
\end{remark}

\subsection{Examples}\label{ss:ex}

Our main running examples of this work are second-order lifts of overdamped Langevin diffusions on $\mathcal{X} = \R^{d}$, given by the solution $(X_t)_{t\geq 0}$ to the stochastic differential equation
\begin{equation}\label{eq:old}
    \dif X_t = -\nabla U(X_t) \dif t + \sqrt{2}\dif B_t.
\end{equation}Here $(B_t)_{t\geq0}$ is a standard $d$-dimensional Brownian motion. The invariant probability measure of \eqref{eq:old} is the Boltzmann-Gibbs measure $\dif \mu(x) \mathrel{\propto} \exp(-U(x)) \dif x$, and the generator of $(P_t)$ in $L^2(\mu)$ can be written as $\Lap = \Delta_x -\nabla U \cdot \nabla _x = -\nabla_x^*\nabla_x$. Therefore, \eqref{eq:pix} essentially amounts to a (super)linear growth of $U$ at infinity, see \cite[Proposition 4.4.2]{bakry_analysis_2014}. Below are some of the dynamics that are covered in this work, which are all second-order lifts of overdamped Langevin dynamics. For examples (i)--(iii) below, we know from \cite{cao2023explicit, lu2022explicit} that the variational hypocoercivity approach is applicable. 

\begin{enumerate}[(i)]
\item \emph{Langevin dynamics} is the solution $(X_t,V_t)_{t\geq 0}$ to the stochastic differential equation
    \begin{equation}\label{eq:ld}
		\left\{
		\begin{aligned}
			\dif X_t &=  V_t\dif t\\
			\dif V_t &=  -\nabla U({X}_t)\dif t - \gamma V_t\dif t + \sqrt{2\gamma}\dif {W}_t,
		\end{aligned}\right.
	\end{equation}
    with a standard $d$-dimensional Brownian motion $(W_t)_{t\geq0}$. 
    Here $\gamma>0$ is the friction parameter, the state spaces are $\calX=\calV=\R^{d}$, and the corresponding operators write $\opT = -v\cdot \nabla_x + \nabla U \cdot \nabla_v$ and $\Lv = -\gamma \nabla_v^*\nabla_v$. In this case, $\kappa = \kappa(\dif v) = \mathcal{N}(0,\mathrm{Id})$ is independent of $x$ and $\hat{\mu} = \mu\otimes \kappa$ is the product measure. Langevin dynamics can also be seen as a perturbation of the deterministic Hamiltonian dynamics  $ \frac{\dif X_t}{\dif t} = V_t , \, \frac{\dif V_t}{\dif t} = -\nabla U(X_t) $ by an Ornstein-Uhlenbeck process with intensity $\gamma$ in the velocity component.
     
\item \emph{Randomised Hamiltonian Monte Carlo} (RHMC) \cite{bou2015randomized}
     is obtained by interspersing the Hamiltonian dynamics $ \frac{\dif X_t}{\dif t} = V_t , \, \frac{\dif V_t}{\dif t} = -\nabla U(X_t) $ with complete velocity refreshments according to $\kappa =\mathcal{N}(0,\mathrm{Id})$ after exponentially distributed waiting times. Here the choice of $\calX$, $\calV$ and $\opT$ is identical to case (i) but $\Lv = \gamma (\Pi-\mathrm{Id})$.
 
\item The \emph{Zig-Zag process} \cite{bierkens2019zig}
    runs in straight lines, interspersed with flips of one of the components of the velocity at random times. Ergodicity of the Zig-Zag process is known without velocity refreshment \cite{bierkens2019ergodicity} but refreshment as in the case of RHMC can still be added to obtain quantitative convergence upper bounds. The operators have the form $\opT = -v\cdot \nabla_x - \sum_{k=1}^d (v_k \partial_{x_k}U)_+ (\mathcal{B}_k - \mathrm{Id}) $ where the $\mathcal{B}_k$ are reflection operators $\mathcal{B}_k f(x,v) =f\bigl(x,v-2v_ke_k\bigr)$ with $e_k$ being the $k$-th cartesian coordinate, and $\Lv = \gamma(\Pi-\mathrm{Id})$. The notable feature in this example is that, unlike the previous two examples, $\opT$ is not antisymmetric in $L^2(\hat{\mu})$. Nevertheless, Assumptions \ref{ass:dmsh3}--\ref{ass:dmsh1} can all be verified. Various other piecewise deterministic Markov processes introduced in the MCMC literature \cite{bouchard2018bouncy, michel2014generalized} also fit into the scope of this framework.
 
\item One example, which this hypocoercivity framework has not yet been applied to, is the \emph{adaptive Langevin} dynamics \cite{jones2011adaptive, ding2014bayesian}
	\begin{equation}
		\label{eq:adl}
		\left\{
		\begin{aligned}
			\dif Q_t &=  V_t\dif t,\\
			\dif V_t &=  -\nabla U(Q_t)\dif t - \frac{Z_t}{\eps} V_t \dif t - \gamma V_t\dif t + \sqrt{2\gamma}\dif {W}_t,\\ \dif Z_t & = \frac{1}{\eps}(|V_t|^2-d) \dif t,
		\end{aligned}\right.
	\end{equation}
    where $(W_t)_{t\geq0}$ is a standard $d$-dimensional Brownian motion.
    The motivation for such dynamics is that in practice, the computation of $\nabla U$ is challenging and may be effectively corrupted, which may lead to a significant bias in the invariant measure when using the Langevin dynamics \eqref{eq:ld}. Hence we define $Z_t$ to be the friction correction variable which is defined by a negative feedback loop control law (as in the Nos\'e--Hoover method \cite{nose1984unified}), so that $|V_t|^2$ is encouraged to stay around $d$, and the true friction $\tfrac{Z_t}{\eps}+\gamma$ may adjust with different noise levels of
    the gradient without knowing the strength of the gradient noise. We refer to \cite{leimkuhler2020hypocoercivity} for a more comprehensive motivation of adaptive Langevin dynamics.

\smallskip

The correct way to view this example is that it is a second-order lift of the overdamped Langevin dynamics jointly in $(q,z)$. In particular, let $x=(q,z) \in \R^{d}\times \R = \mathcal{X}$ and $v\in \R^d = \calV$, the invariant measure of \eqref{eq:adl} writes
     \begin{align*}
    		\dif \mu(q,z) \mathbin{\propto} e^{-U(q) - \frac{|z|^2}{2}}\dif q \dif z, \, \dif \kappa(v) = \frac{1}{(2\pi)^{d/2}} e^{-\frac{\abs{v}^2}{2}} \dif v, \, \dif \hat{\mu}(q,v,z) =\dif \mu(q,z) \otimes \dif\kappa(v).
    	\end{align*}
    The corresponding operators of \eqref{eq:adl} write \cite[(2.5)]{leimkuhler2020hypocoercivity}
    \begin{equation}\label{eq:opL}
        -\opT = \nabla_v^*\nabla_q - \nabla_q^*\nabla_v + \frac{1}{\eps}\oplnh, \qquad \Lv = -\gamma \nabla_v^*\nabla_v
    \end{equation}
    with the $\oplnh$ part resembling the Nos\'e--Hoover feedback control
    \begin{equation}\label{eq:lnh}
        \oplnh = (|v|^2-d) \partial_z - z v\cdot\nabla_v = (\partial_z-\partial_z^*)\nabla_v^*\nabla_v + \Delta_v^*\partial_z - \Delta_v\partial_z^*, \qquad \partial_z^* = -\partial_z + z,  
    \end{equation}
    so that $\oplnh$, and therefore $\opT$, is antisymmetric. 
    A straightforward computation yields
    \[-\Lap = (\opT \Pi)^* (\opT \Pi) = \nabla_q^*\nabla_q + \frac{2d}{\eps^2}\partial_z^*\partial_z,\] which, by a rescaling, is the generator of an overdamped Langevin diffusion in $(q,z)$ variables 
    \[ \dif \begin{pmatrix}
        Q_t \\ Z_t
    \end{pmatrix} = -\begin{pmatrix}
        \nabla U(Q_t) \\ Z_t
    \end{pmatrix} \dif t + \sqrt{2} \dif B_t,\]
    where $(B_t)_{t\geq0}$ is a standard $(d+1)$-dimensional Brownian motion.
    In particular, by tensorisation, the Poincar\'e inequality \eqref{eq:pix} holds true as long as $\nabla_q^*\nabla_q$ has a spectral gap. Although exponential convergence has been shown in \cite{herzog2018exponential, leimkuhler2020hypocoercivity}, in Section \ref{sec:adld} we will apply this framework to obtain sharper convergence upper bounds. Furthermore, we will determine the choice of the parameters $\eps, \gamma$ that optimises the convergence rate.

\item\label{ex:gle} The generalised Langevin dynamics is the Markov process $(X_t,V_t,Z_t)_{t\geq 0}$ on $\calX\times\calV$ with $\calX = \R^d$ and $\calV = \R^d\times\R^d$ given by the solution to the stochastic differential equation
    \begin{equation}
    	\label{eq:gle}
    	\left\{
    	\begin{aligned}
    		\dif X_t &=  V_t\dif t\\
    		\dif V_t &=  -\nabla U(X_t)\dif t +\lambda Z_t\dif t,\\ \dif Z_t & = -\lambda V_t \dif t - \gamma Z_t \dif t + \sqrt{2\gamma}\dif {W}_t,
    	\end{aligned}\right.
    \end{equation}
    where $(W_t)_{t\geq 0}$ is a standard $d$-dimensional Brownian motion and $\lambda,\gamma>0$. It arises as a quasi-Markovian approximation of a more general model for the motion of a
    particle interacting with a heat bath at equilibrium originating in molecular dynamics \cite{pavliotis2021scaling}. It was also considered independently by computer scientists in \cite{mou2021high} as a third-order Langevin dynamics. Equation \eqref{eq:gle} is known as the generalised Langevin equation (GLE), and hypocoercivity for this dynamics has been shown in \cite{ottobre2011asymptotic, pavliotis2021scaling}. Its invariant probability measure is  
    \begin{equation}\label{eq:iv}
        \hat\mu(\dif x\,\dif v\,\dif z) =  \mu(\dif x) \kappa(\dif v\,\dif z)
    \end{equation}
    with $\dif \mu(x) \mathbin{\propto} e^{-U(x)}\dif x$ and $\kappa = \mathcal{N}(0,I_d)\otimes \mathcal{N}(0,I_d)$. The generator of the associated transition semigroup on $L^2(\hat\mu)$ is $\opL = \Lv-\opT$ with the transport and dissipation operators
    \begin{equation*}
        -\opT = v\cdot\nabla_x - \nabla U\cdot\nabla_v + \lambda (z\cdot\nabla_v-v\cdot\nabla_z)
    \end{equation*}
    and
    \begin{equation*}
        \Lv = \gamma(-z\cdot\nabla_z + \Delta_z) = -\gamma \nabla_z^*\nabla_z,
    \end{equation*}
    respectively, where the adjoint is with respect to $\kappa$. In this example, the role of $v$ is played by $(v,z)$, and $\Lv$ in fact only acts on $z$, so that we have an additional level of degeneracy. 

    \smallskip
    
    The work of \cite{cao2023explicit} finds that, despite the degeneracy in ellipticity, Langevin dynamics \eqref{eq:ld} accelerates the convergence rate by (up to a universal constant) a square root if the potential is convex, which is particularly relevant in case $\mathrm{P}_x$ is very small. This raises the question of whether generalised Langevin dynamics, which adds another level of elliptic degeneracy, can converge with even more acceleration than a square root. In Section \ref{sec:gle}, we will give a negative answer to this question by showing the rather surprising observation that GLE is also a second-order lift of an overdamped Langevin diffusion.
    Meanwhile, due to the extra degeneracy in terms of ellipticity compared to examples (i)--(iv), GLE violates Assumption \ref{ass:dmsh1}, since $\Lv$ is only coercive in the $z$-variable. Hence, establishing a sharp upper bound for its convergence rate using the framework covered in this work 
    is much more involved. This is an ongoing work of the authors.

\item Boundary conditions can also be incorporated into this framework. For example, an application to run-and-tumble particles confined by hard walls is given in \cite{eberle2024rtp}. In this case, the domain of $\opT$ incorporates boundary conditions, and the stochastic process associated to $\Lap$ is a sticky Brownian motion. This necessitates a careful treatment of the associated domains in the divergence equation in \Cref{ass:lions}, which we will not cover here.

\item Finally, the equilibrium distribution $\hat\mu$ need not be separable, i.e.\ a product measure. Such situations can for example arise naturally in case $(\calX,g)$ is a Riemannian manifold.
Choosing $\kappa(x,\cdot) = \mathcal{N}(0,g(x))$ allows one to study generalisations of many of the dynamics introduced above to Riemannian manifolds and to interpret them as lifts of an overdamped Riemannian Langevin diffusion, see \cite[Example 4]{eberle2024non} and \cite{eberle2024divergence}. Such processes include, for example, Riemannian Hamiltonian Monte Carlo \cite{GirolamiCalderhead2011Riemann} or Riemannian Langevin dynamics \cite{eberle2024divergence}. Non-separable Hamiltonian may also appear due to physical or numerical purposes, see for example \cite[Section XIII.10]{hairer2006geometric}, \cite{fang2014compressible} and \cite{lelievre2024unbiasing}.

\end{enumerate}

We outline how the rest of the paper is organised: In Section \ref{sec:genfmwk} we present the proof of Theorem \ref{thm:expconv} under the general framework using Assumptions \ref{ass:dmsh3}--\ref{ass:lions}. In Section~\ref{sec:adld}, we provide a new application of the framework to adaptive Langevin dynamics and obtain sharper convergence results compared to  \cite{herzog2018exponential,leimkuhler2020hypocoercivity}. Finally, in Section~\ref{sec:gle}, we present the observation that GLE is also a second-order lift of \eqref{eq:old} and illustrate the resulting lower bound in a benchmark Gaussian case.

\section{The general convergence proof}\label{sec:genfmwk}
The key intermediate step that follows the solution of the divergence equation is the space-time-velocity Poincar\'e-type inequality, which is also known as the averaging lemma. Despite that $\hat{\mu}$ potentially being not separable, it is still convenient to write $\hat{\mu} = \mu \otimes\kappa$ or $\lambda \otimes \hat{\mu} = \bar{\mu} \otimes\kappa$, since we are treating the time and space variables together, and functions may have different regularities in the space-time and velocity variables.
\begin{lemma}\label{lem:avg}
    Under Assumptions \ref{ass:dmsh3}, \ref{ass:dmsh2}, \ref{ass:lions}, for any $f(t,x,v) \in L^2(\bar{\mu}\otimes\kappa)$ with $\iint f  \dif(\bar{\mu}\otimes \kappa)=0$, we have the functional inequality  
    \begin{equation}\label{eq:avg}
        \|\Pi f\|_{L^2(\bar{\mu})} \le C_{1,T}\|(\mathrm{Id}-\Pi)f \|_{L^2(\bar{\mu}\otimes\kappa)} +C_{0,T} \|(\partial_t + \opT)f \|_{L^2(\bar{\mu}; H^{-1}(\kappa))}.
    \end{equation}
    Here $C_{0,T},C_{1,T}$ are the same constants that appear in Assumption \ref{ass:lions}.
\end{lemma}
\begin{proof}
By assumption, we can find $\phi_0,\phi_1$ depending only on $t,x$ which satisfy the divergence equation \eqref{eq:generaldiveq} with $\Pi f$ playing the role of $g$. The key calculation is the following, which, in addition to \cite{cao2023explicit, eberle2024non}, can also be found in \cite[Proposition A.1]{bernard2022hypocoercivity}:
\begin{align*}
    \|\Pi f\|^2_{L^2(\bar{\mu})} & \stackrel{\eqref{eq:generaldiveq}}{=}\big\langle \Pi f,-\partial_t \phi_0 - \Lap \phi_1\big\rangle_{L^2(\bar{\mu})}\\
    & = \big \langle \Pi f, -\partial_t \phi_0 + (\opT\Pi)^*(\opT\Pi) \phi_1 \big \rangle_{L^2(\bar{\mu})} \\ &= \big\langle \Pi f, (-\partial_t + (\opT \Pi)^*)(\phi_0 + \opT \phi_1)\big\rangle_{L^2(\bar{\mu} \otimes\kappa)} \\  &= \big\langle(\partial_t+\opT)\Pi f,\phi_0+\opT\phi_1\big\rangle_{L^2(\bar{\mu} \otimes\kappa)}\\
    & = \big\langle(\partial_t+\opT)f,\phi_0+\opT\phi_1\big\rangle_{L^2(\bar{\mu} \otimes\kappa)}\\
    &\qquad- \big\langle(\partial_t+\opT)(\mathrm{Id}-\Pi)f,\phi_0+\opT\phi_1\big\rangle_{L^2(\bar{\mu} \otimes\kappa)}, \stepcounter{equation} \tag{\theequation} \label{eq:Pivfsqr}
\end{align*}
where the second and third equalities are due to Assumption \ref{ass:dmsh2} and \ref{ass:dmsh3} respectively. The first term on the right-hand side can be controlled using the Cauchy-Schwarz inequality
\begin{align*}
   \big\langle (\partial_t + \opT) f,  \phi_0 + \opT  \phi_1\big\rangle_{L^2(\bar{\mu}\otimes\kappa)}  & \le \|(\partial_t + \opT) f\|_{L^2({\bar{\mu}};H^{-1}(\kappa))}\|\phi_0 + \opT \phi_1\|_{L^2({\bar{\mu}};H^{1}(\kappa))}\\ & \stackrel{\eqref{eq:divbdL2}}{\le} C_{0,T}\|(\partial_t + \opT) f\|_{L^2({\bar{\mu}};H^{-1}(\kappa))}\|\Pi f\|_{L^2(\bar{\mu})}. \stepcounter{equation} \tag{\theequation} \label{eq:avgterm3}
\end{align*}
The second term of \eqref{eq:Pivfsqr} can controlled by
\begin{align*}
     -\big \langle (\partial_t + \opT) (\mathrm{Id}-\Pi)f ,\phi_0 + \opT  \phi_1 \big\rangle_{L^2(\bar{\mu}\otimes\kappa)} & = \big\langle(\mathrm{Id}-\Pi)f ,(\partial_t - \opT^*)(\phi_0 + \opT  \phi_1)\big\rangle_{L^2(\bar{\mu}\otimes\kappa)}  \\ & \le \|(\mathrm{Id}-\Pi)f \|_{L^2(\bar{\mu}\otimes\kappa)} \|(\partial_t - \opT^*)(\phi_0 + \opT  \phi_1) \|_{L^2(\bar{\mu}\otimes\kappa)} \\ & \stackrel{\eqref{eq:divbdH1}}{\le} C_{1,T} \|(\mathrm{Id}-\Pi)f \|_{L^2(\bar{\mu}\otimes\kappa)}\|\Pi f\|_{L^2(\bar{\mu})}. \stepcounter{equation} \tag{\theequation} \label{eq:avgterm1}
\end{align*}
We finish the proof by combining \eqref{eq:Pivfsqr}, \eqref{eq:avgterm3} and \eqref{eq:avgterm1}.
\end{proof}

We are now ready to obtain exponential convergence for solutions of \eqref{eq:genkineq}. The proof of Theorem~\ref{thm:expconv} using Lemma~\ref{lem:avg} is well-established by now, see e.g.\ \cite{albritton2024variational,cao2023explicit,brigati2023construct,eberle2024non}. We include a proof for completeness and the convenience of the reader.
\begin{proof}[Proof of Theorem \ref{thm:expconv}]
   Without loss of generality, we assume $(f_0)=0$, which implies $(f_t)=0$ for any $t>0$ as the equation \eqref{eq:genkineq} preserves mass. Define the time-averaged $L^2$-energy
   \begin{equation}\label{eq:Htau}
       \mathcal{H}(t) = \fint_t^{t+T}\|f(s,\cdot,\cdot)\|^2_{L^2(\hat{\mu})} \dif s.
   \end{equation}
   By a direct energy estimate, using \eqref{eq:optpos}
   \begin{equation}\label{eq:L2engdec}
       \frac{\dif \mathcal{H}(t)}{\dif t} = 2 \fint_t^{t+T} \iint_{\calX \times \calV} f(\Lv - \opT)f \dif\hat\mu (x,v) \dif s\stackrel{\eqref{eq:optpos}}{\le} 2\fint_t^{t+T} \iint_{\calX \times \calV} f \Lv f  \dif \hat{\mu}(x,v) \dif s \le 0.
   \end{equation}
   Meanwhile, applying Lemma \ref{lem:avg} on $f$ between time $t$ and $t+T$ yields
   \begin{align*}
    \fint_t^{t+T}\|\Pi f\|^2_{L^2(\hat{\mu})} \dif s & \stackrel{\eqref{eq:avg}}{\le}  \Big( C_{1,T} \|(\mathrm{Id}-\Pi)f \|_{L^2(\bar{\mu}\otimes\kappa)} +C_{0,T} \|(\partial_t + \opT)f \|_{L^2(\bar{\mu}; H^{-1}(\kappa))}\Big)^2 \\ & \stackrel{\eqref{eq:genkineq},\eqref{eq:piv}, \eqref{eq:H-1kappa}}{\le} -\Big(\frac{C_{1,T}}{\sqrt{\mathrm{P}_v}}+C_{0,T}\sqrt{R}\Big)^2\fint_t^{t+T} \iint_{\calX\times \calV} f\Lv f \dif \hat{\mu} (x,v)\dif s,
   \end{align*}
   which by an $L^2$ orthogonal decomposition leads to
   \begin{align*}
       \fint_t^{t+T} \|f\|^2_{L^2(\hat{\mu})}\dif s &=  \fint_t^{t+T} \|(\mathrm{Id}-\Pi) f\|^2_{L^2(\hat{\mu})}\dif s +  \fint_t^{t+T} \|\Pi f\|^2_{L^2(\hat{\mu})}\dif s \\ & \stackrel{\eqref{eq:piv}}{\le} -\left(\frac{1}{\mathrm{P}_v}+ \Big(\frac{C_{1,T}}{\sqrt{\mathrm{P}_v}}+C_{0,T}\sqrt{R}\Big)^2 \right)\fint_t^{t+T} \iint_{\calX\times \calV} f\Lv f \dif \hat{\mu} (x,v)\dif s.
   \end{align*}
   Substituting this into \eqref{eq:L2engdec}, and we obtain 
   \begin{equation*}
       \frac{\dif \mathcal{H}(t)}{\dif t} \le \frac{-2}{\frac{1}{\mathrm{P}_v}+ \Big(\frac{C_{1,T}}{\sqrt{\mathrm{P}_v}}+C_{0,T}\sqrt{R}\Big)^2} \mathcal{H}(t),
   \end{equation*}
   which directly gives $\mathcal{H}(t) \le e^{-\lambda t}\mathcal{H}(0)$ with $\lambda$ given in \eqref{eq:abstractrate}. We finish the proof as \[\|f(t)\|_{L^2(\hat{\mu})}^2 \le \mathcal{H}(t-T) \le e^{-\lambda(t-T)}\mathcal{H}(0) \le e^{-\lambda(t-T)}\|f_0\|_{L^2(\hat{\mu})}^2.\]
\end{proof}

\section{Example: Hypocoercivity of adaptive Langevin dynamics} \label{sec:adld}
From our discussions in Section \ref{sec:intro}, it suffices to verify Assumption \ref{ass:lions}, since other assumptions are straightforward to check. Furthermore, from \cite{cao2023explicit, brigati2023construct}, it is already clear that the divergence equation \eqref{eq:generaldiveq} has a solution $(\phi_0,\phi_1)$ that satisfies the desired boundary conditions as well as $\phi_0, \partial_t\phi_1, \nabla_x \phi_1 \in H^1(\bar{\mu})$. Hence the estimates \eqref{eq:divbdL2} and \eqref{eq:divbdH1} must hold since their left-hand sides only involve at most first derivatives of $\phi_0$, first and second derivatives of $\phi_1$, and up to eighth moments of $v$ against the standard Gaussian. Hence the bulk of the section is dedicated to computing the constants $C_{0,T}$ and $C_{1,T}$ explicitly.

\smallskip

\begin{assumption}\label{ass:ald}
    Henceforth, the potential $U(q)$ is assumed to be such that 
    \begin{enumerate}
        \item $\nabla^2 U \geq - M \, \mathrm{Id},$ for some $M\geq 0$,
        \item $\Delta U(q)  \leq Ld+ a|\nabla U|^2$, for some $L>0$ and $a\in (0,\frac{1}{2})$,
        \item the spectrum of the operator $\nabla_q^*\nabla_q$ is discrete, with positive spectral gap $\mathrm{P}_q>0$.
    \end{enumerate}
\end{assumption}
\noindent Assumption \ref{ass:ald}, in particular the assumption of discrete spectrum, can be relaxed, as in \cite{brigati2023construct, brigati2024explicit}. The second part of the assumption can be found in various aforementioned works, for example \cite[Assumption 2]{cao2023explicit} and \cite[(3.5)]{bernard2022hypocoercivity}, and is needed for the second inequality of \eqref{eq:zfL2}.

\subsection*{Solving the divergence equation}
In view of Assumption \ref{ass:lions}, we need to solve 
\[
-\partial_t \phi_0 + \left(\nabla_q^*\nabla_q + \frac{2d}{\epsilon^2} \partial_z^* \partial_z\right) \phi_1 = g, 
\]
for $g \in L^2(\bar{\mu})$ with $\iint g(t,x) \, d\bar{\mu}(t,x) =0.$ We adapt the strategy of \cite{eberle2024divergence}, after \cite{cao2023explicit,brigati2023construct}.

\noindent Let $\tdz =cz$, for a constant $c$ to be adjusted, then the original generator reads \[\frac{2d}{\eps^2} (-\partial_{zz} + z \partial_z) = -\frac{2dc^2}{\eps^2}\partial_{\tdz \tdz} + \frac{2d}{\eps^2}\tdz \partial_{\tdz}.\] Since the rescaled equilibrium in $\tdz$ is $e^{-\frac{|\tdz|^2}{2c^2}}$, we know that\footnote{Here $\partial_{\tdz}^\dagger$ denotes the $L^2$-adjoint of $\partial_{\tdz}$ with respect to the tilted measure $e^{-\frac{|\tdz|^2}{2c^2}}$ in the $\tdz$ variable.} $\partial_{\tdz}^\dagger\partial_{\tdz} = -\partial_{\tdz\tdz} + \frac{1}{c^2}\tdz\partial_{\tdz}$, so choosing $c=\frac{\eps}{\sqrt{2d}}$ leads to $\frac{2d}{\eps^2} (-\partial_{zz} + z \partial_z)=\partial_{\tdz}^\dagger\partial_{\tdz}$. We therefore recover the equations as stated in \cite[Theorem 5]{eberle2024divergence} and the following estimates by a simple change of variables.
  
\begin{proposition}
    In the context of \eqref{eq:adl}, for all $g \in L^2(\bar\mu)$ with $\iint g(t,x) \,d\bar\mu =0$, equation \eqref{eq:generaldiveq} admits a solution $(\phi_0,\phi_1)$ with
    \begin{align}
        \label{h1} & \|\phi_0\|^2_{L^2(\bar\mu)} + \|\partial_t \phi_1\|_{L^2(\bar\mu)}^2 + \|\nabla_{q} \phi_1\|_{L^2(\bar\mu)}^2 + \frac{2d}{\epsilon^2} \|\partial_{z} \phi_1\|^2_{L^2(\bar\mu)}\leq c_0 \, \|g\|^2_{L^2(\bar\mu)}\\
        \label{h2} & \|\nabla_{t,q} \phi_0\|^2_{L^2(\bar\mu)} + \|\nabla_{t,q} \nabla_{q} \phi_1\|^2_{L^2(\mu)} + \frac{2d}{\epsilon^2} \left( \|\partial_{z} \phi_0\|^2_{L^2(\bar\mu)} + \|\nabla_{t,q} \partial_z\phi_1\|^2_{L^2(\bar \mu)} + \frac{2d}{\epsilon^2} \|\partial^2_{zz} \phi_1\|^2_{L^2(\bar\mu)} \right) \leq  c_1  \|g\|^2_{L^2(\bar\mu)} 
    \end{align}
    with 
    \begin{equation}\label{poix}
    \begin{aligned}
        &\mathrm{P}_x = \min \left(\mathrm{P}_q, \frac{2d}{\eps^2}\right), \\
        &c_0 = (2 \, T^2 + 43 \, P_{x}^{-1}), 
        &c_1 = \left( 290 + \frac{991}{T^2} \, P_{x}^{-1} + 43 \, \max\left(\mathrm{P}_x^{-1},\frac{T^2}{\pi^2}\right) \, M \right).
    \end{aligned}
    \end{equation}
\end{proposition}

\subsection*{Estimates on the convergence rates}
We start with the moment computations on $\kappa$ which will be used multiple times in the calculations. Since $\int |v|^{2k}\dif\kappa(v) = \Pi_{i=0}^{k-1}(d+2i) = d(d+2)\cdots(d+2k-2)$ for any positive integer $k$, we have
    \begin{align*}
        \int v_1^2 \dif \kappa(v) = 1, \qquad \int v_1^4 \dif \kappa(v)=3, \qquad \int v_1^6 \dif \kappa(v) = 15, \qquad \int v_1^8 \dif \kappa(v)=105, \\ \int (|v|^2-d)^2 \dif \kappa(v) = 2d, \qquad \int (|v|^2-d)^3 \dif \kappa(v) = 8d, \qquad \int (|v|^2-d)^4 \dif \kappa(v) = 12d^2+48d.
    \end{align*}   
Let us first estimate $C_{0,T}$ from the expression in \eqref{eq:divbdL2}: 
\begin{align*}
    &\|\phi_0 + \mathcal T \phi_1\|^2_{L^2(\bar\mu,H^1(\kappa))} = \|\phi_0 + \mathcal T \phi_1\|^2_{L^2(\bar\mu,L^2(\kappa))} + \|\nabla_v (\phi_0 + \mathcal T \phi_1)\|^2_{L^2(\bar\mu,L^2(\kappa))}  \\
    &= \|\phi_0 + v\cdot \nabla_q \phi_1 + \epsilon^{-1} \, (|v|^2-d) \partial_z \phi_1 \|^2_{L^2(\bar\mu,L^2(\kappa))} + \|\nabla_q \phi_1 + 2\epsilon^{-1} v \partial_z \phi_1\|^2_{L^2(\bar\mu,L^2(\kappa))} \\
    & = \|\phi_0\|^2_{L^2(\bar\mu,L^2(\kappa))} + \|v \cdot \nabla_q \phi_1\|^2_{L^2(\bar\mu,L^2(\kappa))} + \frac{1}{\epsilon^2}\|(|v|^2-d) \partial_z \phi_1\|^2_{L^2(\bar\mu,L^2(\kappa))} \\
    & \quad + \|\nabla_q\phi_1\|^2_{L^2(\bar\mu,L^2(\kappa))} + \frac{4}{\epsilon^2} \|v \partial_z \phi_1\|^2_{L^2(\bar\mu,L^2(\kappa))} \\
    &= \|\phi_0\|^2_{L^2(\bar\mu)} + \|\nabla_q \phi_1\|^2_{L^2(\bar\mu)} + \frac{2d}{\epsilon^2} \|\partial_z \phi_1\|^2_{L^2(\bar\mu)} + \|\nabla_q \phi_1\|^2_{L^2(\bar\mu)} + \frac{4d}{\epsilon^2} \|\partial_z \phi_1\|^2_{L^2(\bar\mu)} \\ 
    &\leq \,2 c_0 \|g\|^2_{L^2(\bar \mu)}.
\end{align*}
In the same way, for $C_{1,T}$ using \eqref{eq:divbdH1}, we have to bound
    $\|(\partial_t + \mathcal T) (\phi_0 + \mathcal T \phi_1)\|^2_{L^2(\bar\mu \otimes \kappa)}$ as $\opT$ is antisymmetric.
Note that \[ (\partial_t + \opT)(\phi_0 + \opT \phi_1) = \mathrm{I} + \mathrm{II} + \mathrm{III} + \mathrm{IV}\] with 
    \begin{align*}
        \mathrm{I} & = \partial_t \phi_0 + v^\top (\nabla^2_q  \phi_1) v - \nabla U \cdot\nabla_q \phi_1 -\frac{2d}{\epsilon^2}z\partial_z \phi_1 +  \epsilon^{-2}(|v|^2-d)^2 \partial^2_{zz} \phi_1   \\ 
    \mathrm{II} &=  v \cdot \nabla_q (\partial_t\phi_1+\phi_0)-\frac{2}{\epsilon}\nabla U \cdot v \partial_z \phi_1-\frac{1}{\epsilon} zv\cdot\nabla_q  \phi_1  \\
    \mathrm{III} &= \epsilon^{-1} \, (|v|^2-d) (\partial_z \partial_t \phi_1 + \partial_z \phi_0)-\frac{2}{\eps^2}(|v|^2-d) z\partial_z\phi_1 \\
    \mathrm{IV} &=  2 \epsilon^{-1} v (|v|^2-d) \cdot \nabla_q \partial_z\phi_1 
    \end{align*}
Since $\mathrm{II}+\mathrm{IV}$ is orthogonal to both $\mathrm{I}$ and  $\mathrm{III}$ in $L^2(\bar\mu)$, we have
\[ \|(\partial_t + \mathcal T) (\phi_0 + \mathcal T \phi_1)\|^2_{L^2(\bar\mu \otimes \kappa)} \leq 2\left(\|\mathrm{I}\|^2_{L^2(\bar\mu \otimes \kappa)} + \|\mathrm{II}\|^2_{L^2(\bar\mu \otimes \kappa)}+\|\mathrm{III}\|^2_{L^2(\bar\mu \otimes \kappa)}+\|\mathrm{IV}\|^2_{L^2(\bar\mu \otimes \kappa)}\right).\] Using \cite[Lemma 2.2]{cao2023explicit}, we have that
\begin{equation}\label{eq:zfL2}
        \|zf\|_{L^2(\bar{\mu})}^2 \le 4\|\partial_z f\|_{L^2(\bar{\mu})}^2 + 2 \|f\|_{L^2(\bar{\mu})}^2, \, \qquad  \|f\nabla U\|_{L^2(\bar\mu)}^2 \le 16\|\nabla_q f \|_{L^2(\bar\mu)}^2+4Ld\| f \|_{L^2(\bar\mu)}^2.
    \end{equation} 
These allow us to estimate directly 
    \begin{align*}
        \|\mathrm{IV}\|_{L^2(\bar\mu \otimes \kappa)}^2 & = \frac{4}{\epsilon^2} \|\nabla_q \partial_z \phi_1\|_{L^2(\bar{\mu})}^2 \int v_1^2(|v|^2-d)^2 \dif \kappa(v) = \frac{8(d+4)}{\epsilon^2} \|\nabla_q \partial_z \phi_1\|_{L^2(\bar{\mu})}^2. \\ 
        \|\mathrm{III}\|_{L^2(\bar\mu \otimes \kappa)}^2 & = \frac{2d}{\eps^2} \|\partial_z \partial_t \phi_1 + \partial_z \phi_0-\frac{2}{\eps}z\partial_z\phi_1 \|_{L^2(\bar{\mu})}^2 \\ & \le \frac{6d}{\eps^2}(\|\partial_z \partial_t \phi_1\|_{L^2(\bar{\mu})}^2 + \|\partial_z \phi_0\|_{L^2(\bar{\mu})}^2 + \frac{4}{\eps^2}\|z\partial_z \phi_1\|_{L^2(\bar{\mu})}^2) \\ & \stackrel{\eqref{eq:zfL2}}{\le} \frac{6d}{\eps^2}\Big(\|\partial_z \partial_t \phi_1\|_{L^2(\bar{\mu})}^2 + \|\partial_z \phi_0\|_{L^2(\bar{\mu})}^2 + \frac{8}{\eps^2}\|\partial_z\phi_1\|_{L^2(\bar{\mu})}^2 + \frac{16}{\eps^2}\|\partial^2_{zz}\phi_1\|_{L^2(\bar{\mu})}^2 \Big). \\ 
        \|\mathrm{I}\|^2_{L^2(\bar{\mu})} & = \|\partial_t \phi_0 - \nabla U \cdot\nabla_q \phi_1 -\frac{2d}{\epsilon^2}z\partial_z \phi_1\|^2_{L^2(\bar{\mu})}+ \|v^\top \nabla_{q}^2 \phi_1 v\|^2_{L^2(\bar{\mu})}  \\ &\quad+ \frac{1}{\eps^4} \|(|v|^2-d)^2 \partial^2_{zz} \phi_1\|^2_{L^2(\bar{\mu})} \\ & \quad + 2 \Big \langle \partial_t \phi_0 - \nabla U \cdot\nabla_q \phi_1 -\frac{2d}{\epsilon^2}z\partial_z \phi_1,v^\top \nabla_{q}^2 \phi_1 v + \frac{1}{\eps^2}(|v|^2-d)^2 \partial^2_{zz} \phi_1 \Big \rangle \\ & \quad + 2 \Big\langle v^\top \nabla_{q}^2 \phi_1 v , \frac{1}{\eps^2}(|v|^2-d)^2 \partial^2_{zz} \phi_1\Big \rangle \\ & = \|\partial_t \phi_0 - \nabla U \cdot\nabla_q \phi_1 -\frac{2d}{\epsilon^2}z\partial_z \phi_1\|^2_{L^2(\bar{\mu})}+ 2\| \nabla_{q}^2 \phi_1 \|^2_{L^2(\bar{\mu})} +\|\Delta_q \phi_1\|^2_{L^2(\bar{\mu})} \\ & \quad + \frac{12d(d+4)}{\eps^4} \| \partial^2_{zz} \phi_1\|^2_{L^2(\bar{\mu})} + \frac{4(d+4)}{\eps^2} \Big\langle \Delta_q \phi_1, \partial_{zz}^2\phi_1\Big \rangle \\ & \quad + 2 \Big \langle \partial_t \phi_0 - \nabla U \cdot\nabla_q \phi_1 -\frac{2d}{\epsilon^2}z\partial_z \phi_1,\Delta_q \phi_1  + \frac{2d}{\eps^2} \partial^2_{zz} \phi_1 \Big \rangle \\ & = \|\partial_t \phi_0 - \nabla U \cdot\nabla_q \phi_1 -\frac{2d}{\epsilon^2}z\partial_z \phi_1+ \Delta_q \phi_1 + \frac{2d}{\eps^2}\partial_{zz}^2\phi_1\|^2_{L^2(\bar{\mu})} +2\| \nabla_{q}^2 \phi_1 \|^2_{L^2(\bar{\mu})} \\ & \quad + \frac{16}{\eps^2} \Big\langle \Delta_q \phi_1, \partial_{zz}^2\phi_1\Big \rangle + \frac{8d(d+6)}{\eps^4} \| \partial^2_{zz} \phi_1\|^2_{L^2(\bar{\mu})}\\ &\le \|g\|^2_{L^2(\bar{\mu})} +10\| \nabla_{q}^2 \phi_1 \|^2_{L^2(\bar{\mu})} + \frac{8d(d+7)}{\eps^4} \| \partial^2_{zz} \phi_1\|^2_{L^2(\bar{\mu})}.
        \end{align*}      
      \begin{align*}
        \|\mathrm{II}\|^2_{L^2(\bar\mu)} &= \|\nabla_q (\partial_t \phi_1 + \phi_0) - \frac{2}{\epsilon} \nabla U \partial_z \phi_1 - \frac{1}{\epsilon} z \nabla_q \phi_1\|^2_{L^2(\bar \mu)} \\ &\leq 4 \left( \|\nabla_q\partial_t \phi_1\|^2_{L^2(\bar\mu)} + \|\nabla_q \phi_0\|^2_{L^2(\bar\mu)} + \frac{4}{\epsilon^2} \|\nabla U \partial_z \phi_1\|^2_{L^2(\bar\mu)} + \frac{1}{\epsilon^2} \|z\nabla_q\phi_1\|^2_{L^2(\bar\mu)}  \right) \\ & \stackrel{\eqref{eq:zfL2}}{\leq} 4\left( \|\nabla_q\partial_t \phi_1\|^2_{L^2(\bar\mu)} + \|\nabla_q \phi_0\|^2_{L^2(\bar\mu)} + \frac{68}{\eps^2} \|\nabla_q \partial_z \phi_1\|^2_{L^2(\bar\mu)} + \frac{16Ld}{\eps^2}\, \|\partial_z \phi_1\|^2_{L^2(\bar\mu)}   +\frac{2}{\eps^2}\|\nabla_q \phi_1\|^2_{L^2(\bar\mu)} \right). 
    \end{align*}
To combine everything above, using $d\ge 1$, we arrive at
\begin{align*}
    \|(\partial_t + &\mathcal T)  (\phi_0 + \mathcal T \phi_1)\|^2_{L^2(\bar\mu \otimes \kappa)} \\ & \leq \frac{16(d+38)}{\epsilon^2} \|\nabla_q \partial_z \phi_1\|_{L^2(\bar{\mu})}^2 + \frac{12d}{\eps^2}\Big(\|\partial_z \partial_t \phi_1\|_{L^2(\bar{\mu})}^2 + \|\partial_z \phi_0\|_{L^2(\bar{\mu})}^2 \Big) + 2 \|g\|^2_{L^2(\bar{\mu})} +20\| \nabla_{q}^2 \phi_1 \|^2_{L^2(\bar{\mu})} \\ & \quad  +\frac{32d}{\eps^4}(3+4\eps^2L)\|\partial_z\phi_1\|_{L^2(\bar{\mu})}^2+ \frac{16d(d+19)}{\eps^4} \| \partial^2_{zz} \phi_1\|^2_{L^2(\bar{\mu})} \\ & \quad + 8\Bigl(\|\nabla_q\partial_t \phi_1\|^2_{L^2(\bar\mu)} + \|\nabla_q \phi_0\|^2_{L^2(\bar\mu)}+\frac{2}{\eps^2}\|\nabla_q \phi_1\|^2_{L^2(\bar\mu)}\Bigr) \\ & \stackrel{\eqref{h1},\eqref{h2}}{\le} 314\Big(c_1 + (\frac{1}{\eps^2}+L)c_0\Big)\|g\|^2_{L^2(\bar{\mu})}.
\end{align*}

\subsection*{Optimisation of the parameters}
Now that we have $C_{0,T}^2 = 2c_0$ and $C_{1,T}^2 = 314(c_1 + (\frac{1}{\eps^2}+L)c_0)$ with $c_0,c_1$ given in \eqref{poix}, and notice that Assumption \ref{ass:dmsh1} holds with $R =\mathrm{P}_v = \gamma$ in our example, we can directly apply Theorem \ref{thm:expconv} and see that the convergence rate
\[\lambda \ge \frac{2\gamma}{1+(C_{1,T}+C_{0,T}\gamma)^2}\]
for any $T$. From the expressions in \eqref{poix}, it is advantageous to pick $T^2 \sim \mathrm{P}_x^{-1}$, since the scaling of the rate will worsen in either regime $T^2\gg \mathrm{P}_x^{-1}$ or $T^2 \ll \mathrm{P}_x^{-1}$. For simplicity we pick $T^2 = \pi^2 \mathrm{P}_x^{-1}$, which yields \[c_0\le 63\mathrm{P}_x^{-1}, c_1 \le 391+43\mathrm{P}_x^{-1}M, C_{0,T}^2 \le 126\mathrm{P}_x^{-1}, C_{1,T}^2 \le 314\Big(391+43\mathrm{P}_x^{-1}M + 63(\frac{1}{\eps^2}+L)\mathrm{P}_x^{-1}\Big),  \]
and consequently, writing $\mathrm{P}_x^{-1} =\max\{\mathrm{P}_q^{-1},\tfrac{\eps^2}{2d}\} \le \mathrm{P}_q^{-1}+ \tfrac{\eps^2}{2d}$, we obtain \[ \lambda \ge 
\frac{2\gamma}{1+\tfrac{3}{2} C_{1,T}^2 + 3C_{0,T}^2\gamma^2} \ge \frac{2\gamma}{61388 + (\mathrm{P}_q^{-1}+ \tfrac{\eps^2}{2d})\Big(378\gamma^2 + 6751M + 9891(\frac{1}{\eps^2}+L)\Big) }.\]
From here, we already recover the scaling\footnote{The scaling in the published version \cite[(2.12)]{leimkuhler2020hypocoercivity} was incorrect. It was corrected later in the third arXiv version of their work \cite[(2.12)]{leimkuhler2019hypov3}.}  \[\lambda \ge \bar\lambda\min\left\{\frac{1}{\gamma}, \frac{1}{\gamma \eps^2}, \gamma \eps^2, \frac{\gamma}{\eps^2}\right\}\]
in \cite{leimkuhler2019hypov3} for fixed $M$, $d$, $L$ and $\mathrm{P}_q$.
For optimal scaling of $\lambda$, we pick $\eps^2 = \sqrt{\frac{d}{\mathrm{P}_q(M+L+\gamma^2)}}$ and then $\gamma = \sqrt{\mathrm{P}_q+M+L}$, and obtain the rate \[ \lambda\ge \frac{\mathrm{P}_q}{66334\sqrt{\mathrm{P}_q+M+L}},\] which is dimension independent. For convex potentials that are not far from quadratic, we have $M=0$ and $L\sim \mathrm{P}_q$, and therefore $\lambda \sim \sqrt{\mathrm{P}_q}$, which means up to a universal constant, adaptive Langevin dynamics is also an optimal second-order lift of the overdamped Langevin dynamics.

\begin{remark}
    We are unable to completely recover the result in \cite{cao2023explicit} for Langevin dynamics \eqref{eq:ld} in the sense that the rates in \cite{cao2023explicit} do not depend on $L$ that appears in Assumption \ref{ass:ald} (2). The bottleneck is the estimate of $\|\mathrm{II}\|_{L^2(\bar\mu)}$, especially the contribution from $\nabla U \partial_z \phi_1$, where a factor of $L$ arises after using \eqref{eq:zfL2}, while for the case of \eqref{eq:ld}, any term with $\nabla U$ only appears in a term similar to $(\mathrm{I})$, which can be simplified using the divergence equation \eqref{eq:generaldiveq}. It is unclear to us whether this is only an artifact of our proof techniques.
\end{remark}

\section{Lower bounds for the generalised Langevin equation}\label{sec:gle}

Recall the generalised Langevin dynamics $(X_t,V_t,Z_t)_{t\geq 0}$ on $\calX\times\calV$ with $\calX=\R^d$ and $\calV=\R^d\times\R^d$, which is the solution to the GLE \eqref{eq:gle}.
 Its invariant probability measure is  
\begin{equation*}
    \hat\mu(\dif x\,\dif v\,\dif z) =  \mu(\dif x) \kappa(\dif v\,\dif z)
\end{equation*}
with $\mu(\dif x) \mathbin{\propto} e^{-U(x)}\dif x$ and $\kappa = \mathcal{N}(0,I_d)\otimes \mathcal{N}(0,I_d)$. The generator of the associated transition semigroup on $L^2(\hat\mu)$ is $\opL = \Lv-\opT$ with the antisymmetric part
\begin{equation*}
 -\opT = v\cdot\nabla_x - \nabla U\cdot\nabla_v + \lambda (z\cdot\nabla_v-v\cdot\nabla_z)
\end{equation*}
and the symmetric part
\begin{equation*}
 \Lv = \gamma(-z\cdot\nabla_z + \Delta_z) =- \gamma \nabla_z^*\nabla_z,
\end{equation*}
which are the transport and dissipation operators, respectively. We correspondingly set $\Pi f(x) = \int_{\R^d\times\R^d}f(x,v,z)\dif\kappa(v,z)$. 
Note that $\Lv$ acts only on the $z$-component and not on $v$. 

\smallskip

A key observation is that the generalised Langevin dynamics is a second-order lift of an overdamped Langevin diffusion. This may seem counterintuitive at first, since the degeneracy is of higher order in terms of the number of Lie brackets needed to span the whole space. This corresponds to the noise being more degenerate than e.g.\ in the case of Langevin dynamics, since it is only present in the second derivative $Z$ of the spatial component $X$. The main difficulty in the treatment of this generalised case is precisely the fact that the ellipticity of $\Lv$ is only present in $z$, whereas the lifting variables are $(v,z)$. In particular, Assumption~\ref{ass:dmsh1} is violated, since $\ker(\Lv)\neq\mathrm{Im}(\Pi)$.

\begin{lemma}\label{lem:GLElift}
    Assume that $\mu$ satisfies a Poincar\'e inequality with constant $\mathrm{P}_x^{-1}$, i.e.\ 
    \begin{equation*}
        \int_{\R^d}f^2\dif\mu \leq \mathrm{P}_x^{-1}\int_{\R^d}|\nabla_x f|^2\dif\mu
    \end{equation*}
    for all $f\in H^{1}(\mu)$ with $\int_{\R^d}f\dif\mu = 0$.
    Then Assumptions \ref{ass:dmsh3} and \ref{ass:dmsh2} are satisfied. In particular, generalised Langevin dynamics is a second-order lift of the overdamped Langevin diffusion with generator $\frac{1}{2}\Lap = -\frac{1}{2}\nabla U\cdot\nabla_x+\frac{1}{2}\Delta_x$.
\end{lemma}
\begin{proof}
    It is immediate that
    \begin{equation*}
        \Pi\opT\Pi f = -\int_{\R^d\times\R^d}v\cdot\nabla_x\Pi f(x)\dif\kappa(v,z) = 0,
    \end{equation*}
    so that \Cref{ass:dmsh3} is satisfied. Furthermore,
    \begin{equation*}
        -(\opT\Pi)^*(\opT\Pi) f = \int_{\R^d\times\R^d} v^\top(\nabla_x^2\Pi f)v-\nabla U\cdot\nabla_x\Pi f\dif\kappa(v,z) = \Lap\Pi f,
    \end{equation*}
    which, together with the Poincar\'e inequality for $\mu$, yields \Cref{ass:dmsh2}.
\end{proof}

The following lower bound is now an immediate consequence of the general lower bound for second-order lifts in \Cref{rem:lowerbound}.

\begin{corollary}[Lower bound on the relaxation time]\label{thm:lowerboundGLE}
    The relaxation time $t_\rel(\hat P)$ of the generalised Langevin dynamics satisfies
    \begin{equation*}
        t_\rel(\hat P) \geq \frac{1}{2}\,\mathrm{P}_x^{-1/2}\,.
    \end{equation*}
\end{corollary}

In particular, compared to the overdamped Langevin diffusion, the convergence to equilibrium of the generalised Langevin dynamics as measured by the relaxation time is not accelerated by more than a square root. The square-root speed-up to a relaxation time of the order $\mathrm{P}_x^{-1/2}$ compared to the order $\mathrm{P}_x^{-1}$ for the overdamped Langevin diffusion is already achieved by the (second-order) Langevin dynamics \cite{cao2023explicit}.

\begin{example}[Gaussian case]
    In order to see how sharp the lower bound is, let us consider the simple case where $\mu=\mathcal{N}(0,m^{-1})$ is a centred normal distribution on $\R$ with variance $m^{-1}$. In this case, \Cref{ass:dmsh2} is satisfied with $\mathrm{P}_x = m$. The GLE \eqref{eq:gle} then becomes an Ornstein-Uhlenbeck process
    \begin{equation*}
        \dif\begin{pmatrix}X_t\\V_t\\Z_t\end{pmatrix} = A\begin{pmatrix}X_t\\V_t\\Z_t\end{pmatrix}\dif t + \Sigma\dif W_t
    \end{equation*}
    with 
    \begin{equation*}
        A = \begin{pmatrix}0&1&0\\-m&0&\lambda\\0&-\lambda&-\gamma\end{pmatrix}\qquad\textup{and}\qquad\Sigma = \begin{pmatrix}0\\0\\\sqrt{2\gamma}\end{pmatrix}\,.
    \end{equation*}
    The asymptotic decay rate of the associated transition semigroup $(\hat P_t)_{t\geq0}$ is given by the spectral gap of its generator in $L^2(\hat\mu)$, which by \cite{metafune2002spectrum} coincides with the spectral gap
    \begin{equation*}
        \gap(A) = \inf \{-\Re(\lambda)\colon\lambda\in\spec(A)\}
    \end{equation*}
    of $A$.
    When fixing the scaling $\lambda = a\sqrt{m}$ and $\gamma = b\sqrt{m}$ for some $a,b>0$, an explicit computation shows that the spectrum of $A$ is given by
    \begin{align*}
            \mu_1 &= \sqrt{m}\left(-\frac{b}{3}-\frac{\beta}{{\left(\sqrt{\alpha^2+\beta^3}-\alpha\right)}^{1/3}}+\left(\sqrt{\alpha^2+\beta^3}-\alpha\right)^{1/3}\right)\,,\\
            \mu_2& = \sqrt{m}\left(-\frac{b}{3}+\frac{(1-\mathrm{i}\,\sqrt{3})\beta}{2{\left(\sqrt{\alpha^2+\beta^3}-\alpha\right)}^{1/3}}-\frac{(1+\mathrm{i}\,\sqrt{3}){\left(\sqrt{\alpha^2+\beta^3}-\alpha\right)}^{1/3}}{2}\right)\,,\\ 
            \mu_3& = \sqrt{m}\left(-\frac{b}{3}+\frac{(1+\mathrm{i}\,\sqrt{3})\beta}{2{\left(\sqrt{\alpha^2+\beta^3}-\alpha\right)}^{1/3}}-\frac{(1-\mathrm{i}\,\sqrt{3}){\left(\sqrt{\alpha^2+\beta^3}-\alpha\right)}^{1/3}}{2}\right)\,,
    \end{align*}
    where
    \begin{align*}
        \alpha &= \frac{b}{2}+\frac{b^3}{27}-\frac{b\left(a^2+1\right)}{6}\,,\\
        \beta &= \frac{a^2}{3}-\frac{b^2}{9}+\frac{1}{3}\,.
    \end{align*}
    The spectral gap is maximised when all three eigenvalues coincide, which is the case for $\alpha = \beta = 0$, or equivalently $a = 2\sqrt{2}$ and $b = 3\sqrt{3}$. Correspondingly, the choice $\lambda = 2\sqrt{2m}$ and $\gamma = 3\sqrt{3m}$ yields the optimal spectral gap $\sqrt{3m}$.

    \smallskip

    To obtain non-asymptotic bounds, the spectral gap does not suffice, and we consider the relaxation time.
    By \cite{Arnold2022Propagator}, the transition semigroup $(\hat P_t)_{t\geq0}$ satisfies
    \begin{equation*}
        \norm{\hat P_t}_{L_0^2(\hat\mu)\to L_0^2(\hat\mu)} = \norm{\exp(t\tilde A)}\qquad\textup{for all }t\geq0
    \end{equation*}
    with
    \begin{equation*}
        \tilde A = \begin{pmatrix}0&\sqrt{m}&0\\-\sqrt{m}&0&\lambda\\0&-\lambda&-\gamma\end{pmatrix}= \sqrt{m}\begin{pmatrix}0&1&0\\-1&0&2\sqrt{2}\\0&-2\sqrt{2}&-3\sqrt{3}\end{pmatrix}
    \end{equation*}
    with the choices of $\lambda$ and $\gamma$ given above.
    Since $\spec(\tilde A)=\{-\sqrt{3m}\}$ and $\dim\ker(\tilde A+\sqrt{3m}I)=1$, we obtain
    \begin{equation*}
        \exp(t\tilde A) = e^{-\sqrt{3m}t}S^{-1}\begin{pmatrix}1 & \sqrt{m}t & \frac{1}{2}mt^2 \\ 0 & 1 & \sqrt{m}t \\ 0 & 0 & 1\end{pmatrix}S\,.
    \end{equation*}
    with $S\in\mathrm{GL}(3,\R)$. An explicit calculation of the singular values then shows that 
    \begin{equation*}
        \norm{\hat P_t}_{L_0^2(\hat\mu)\to L_0^2(\hat\mu)} = e^{-\sqrt{3m}t}p(\sqrt{m}t)^{1/2}\quad\textup{with }p(s) = 1 + s^2+\frac{1}{8}s^4+\left(2s+\frac{1}{2}s^3\right)\sqrt{\frac{s^2}{16}+\frac{1}{2}}\,.
    \end{equation*}
    A numerical calculation yields the relaxation time
    \begin{equation*}
        t_\rel(\hat P) = \inf\{t\geq 0\colon \norm{\hat P_t}_{L_0^2(\hat\mu)\to L_0^2(\hat\mu)}\leq e^{-1}\}\approx0.964m^{-1/2}
    \end{equation*}
    of the generalised Langevin dynamics. By a factorisation argument, the same holds true for arbitrary Gaussian probability measures on $\R^d$. Comparing this to the lower bound $0.5m^{-1/2}$ from \Cref{thm:lowerboundGLE}, this shows that, in the terminology of \cite{eberle2024non}, on the class of quadratic potentials on $\R^d$, the generalised Langevin dynamics with optimal choice of $\lambda$ and $\gamma$ is a $1.93$-optimal lift of the corresponding overdamped Langevin diffusion. As a comparison, on the same class of potentials, Langevin dynamics \eqref{eq:ld} with optimal choice of friction $\gamma$ achieves a relaxation time of approximately $2.73m^{-1/2}$ and is thus only a $5.46$-optimal lift, see \cite[Example 14]{eberle2024non}.
    
\end{example}

\section*{Acknowledgements}
We would like to thank Andreas Eberle and Gabriel Stoltz for many helpful discussions. GB has received funding from the European Union Horizon 2020 research and innovation programme under the Marie Sklodowska-Curie grant agreement No 101034413. FL wurde gef\"ordert durch die Deutsche Forschungsgemeinschaft (DFG) im Rahmen der Exzellenzstrategie des Bundes und der L\"ander -- GZ2047/1, Projekt-ID 390685813. LW is supported by the National Science Foundation via grant DMS-2407166. He is also indebted to the Mathematical Sciences department at Carnegie Mellon University for partly supporting his visit to Europe in July 2024. Part of this work was completed when GB and LW were visiting the Institute for Applied Mathematics in Bonn. GB and LW would like to thank IAM for their hospitality. 
\bibliographystyle{plain}
\bibliography{main}
\end{document}